\documentclass[11pt,oneside]{amsart}
\usepackage{amssymb,amsmath,amsfonts,latexsym,amsthm,geometry,graphicx}
\usepackage{amsmath}
\usepackage{amsfonts}
\usepackage{color, soul}
\usepackage{amssymb}
\usepackage{graphicx}
\usepackage[all]{xy}

\newtheorem{theorem}{Theorem}[section]
\newtheorem{corollary}[theorem]{Corollary}
\newtheorem{proposition}[theorem]{Proposition}
\newtheorem{lemma}[theorem]{Lemma}

\theoremstyle{definition}

\newtheorem{remark}[theorem]{Remark}
\newtheorem{definition}[theorem]{Definition}

\begin{document}
\title{An index of summability for pairs of Banach spaces}
\author[Mariana Maia]{M. Maia}
\address{Departamento de Matem\'{a}tica\\
Universidade Federal da Para\'{\i}ba \\
58.051-900 - Jo\~{a}o Pessoa, Brazil.}
\email{mariana.britomaia@gmail.com }
\author[Daniel Pellegrino]{D. Pellegrino}
\address{Departamento de Matem\'{a}tica \\
Universidade Federal da Para\'{\i}ba \\
58.051-900 - Jo\~{a}o Pessoa, Brazil.}
\email{dmpellegrino@gmail.com }
\author[Joedson Santos]{J. Santos}
\address{Departamento de Matem\'{a}tica \\
Universidade Federal da Para\'{\i}ba \\
58.051-900 - Jo\~{a}o Pessoa, Brazil.}
\email{joedsonmat@gmail.com}
\thanks{2010 Mathematics Subject Classification: }
\thanks{}
\maketitle

\begin{abstract}
We introduce the notion of index of summability for pairs of Banach spaces; for Banach spaces $E,F$, this index plays the
role of a kind of measure of how the $m$-homogeneous polynomials from $E$ to $F$ are far from being absolutely summing. In some
cases the optimal index of summability is computed.
\end{abstract}

\tableofcontents

\section{Introduction and background}

For $1\leq q\leq p<\infty $ and Banach spaces $E,F$ over $\mathbb{K}=\mathbb{%
R}$ or $\mathbb{C},$ we recall that a continuous linear operator $%
u:E\rightarrow F$ is absolutely $(p,q)$-summing if there is a constant $%
C\geq 0$ such that
\begin{equation}
\left( \sum_{k=1}^{n}\left\Vert u(x_{k})\right\Vert ^{p}\right) ^{\frac{1}{p}%
}\leq C\sup_{\varphi \in B_{E^{\ast }}}\left( \sum_{k=1}^{n}\left\vert
\varphi (x_{k})\right\vert ^{q}\right) ^{\frac{1}{q}}  \label{pq}
\end{equation}%
for every $n\in \mathbb{N}$ and $x_{1},...,x_{n}\in E.$ Above, and from now
on the topological dual of $E$ and its closed unit ball are denoted by $%
E^{\ast }$ and $B_{E^{\ast }}$, respectively.

The space of absolutely $(p,q)$-summing linear operators from $E$ to $F$ is
denoted by $\Pi _{(p,q)}\left( E;F\right) $. The $(p,q)$-summing norm of $u$%
, defined as the infimum of the constants $C$ in (\ref{pq}), is represented
by $\pi _{p,q}(u)$. If $p=q$ the operator $u$ is simply called absolutely $p$%
-summing and write $\Pi _{p}\left( E;F\right) $ and $\pi _{p}(u)$ for the
space of absolutely $p$-summing operators and the $p$-summing norm of $u$,
respectively. For the theory of absolutely summing operators we refer to
\cite{Diestel}.

When only sequences $(x_{j})_{j=1}^{n}$ of fixed length $n$ are considered,
the infimum over all $C$ satisfying (\ref{pq}) is denoted by $\pi
_{p,q}^{(n)}(u)$ (or $\pi _{p}^{(n)}(u)$ when $p=q$). Of course, $\pi
_{p,q}^{(n)}(u)\leq \pi _{p,q}(u)$. In \cite{Szarek, TJ} the authors
investigated in depth estimates of the type
\begin{equation*}
\pi _{p,q}(u)\leq c\pi _{p,q}^{(n)}(u),
\end{equation*}%
where $c$ is a positive constant. These estimates show that the $(p,q)$%
-summing norm of an operator can be sometimes well-approximated using only
\textquotedblleft few\textquotedblright\ vectors in the definition of the $%
(p,q)$-summing norm. The following results of finite-dimensional nature will
be crucial in this paper:

\begin{theorem}
\bigskip (Szarek \cite{Szarek}) \label{szarek} There exists a universal
constant $C$ such that whenever $u:E\rightarrow F$ ($E,F$ are Banach spaces)
is a finite rank linear operator (say $\mathrm{rank}(u)=n$) and $q \geq 2,$ then
\begin{equation*}
\pi _{q,2}(u)\leq C\pi _{q,2}^{(n)}(u).
\end{equation*}
\end{theorem}

\begin{theorem}
(K\"{o}nig, Retherford, Tomczak-Jaegermann \cite{Konig}) \label{konig} Let $%
id_{X_{n}}$ denote the identity on a $n$-dimensional space $X_{n}$. For $q>2$%
, we have
\begin{equation*}
(2e)^{-1}n^{\frac{1}{q}}\leq \pi _{q,2}(id_{X_{n}}).
\end{equation*}
\end{theorem}

From now on, as usual, given $x_{1},...,x_{n}\in E$, we define
\begin{equation*}
\Vert (x_{k})_{k=1}^{n}\Vert _{w,p}:=\sup_{\varphi \in B_{E^{\ast }}}\left(
\sum_{k=1}^{n}\left\vert \varphi (x_{k})\right\vert ^{p}\right) ^{\frac{1}{p}%
}.
\end{equation*}

Let $m\in \mathbb{N}$ and $E_{1},...,E_{m}$ be Banach spaces over $\mathbb{K}
$. By $\mathcal{L}\left( E_{1},\dots ,E_{m};F\right) $ we denote the Banach
space of all bounded $m$-linear operators from $E_{1}\times \cdots \times
E_{m}$ into $F$. In the case $E_{1}=\cdots =E_{m}=E$, we will simply write $%
\mathcal{L}\left( ^{m}E;F\right) $, whereas $\mathcal{L}\left( E;F\right) $
is the usual Banach space of all continuous linear operators from $E$ to $F$%
. For the theory of multilinear operators and polynomials between Banach
spaces we refer to the excellent books of Dineen \cite{dineen} and Mujica
\cite{mujica}.

For $1\leq q\leq p<\infty $, an $m$-linear operator $T\in \mathcal{L}%
(E_{1},...,E_{m};F)$ is called \textit{multiple} $\left( p,q\right) $\textit{%
-summing} (\cite{collec, pgtese}) if there is a constant $C\geq 0$ such that
\begin{equation}
\left( \sum_{k_{1},...,k_{m}=1}^{n}\left\Vert T\left(
x_{k_{1}}^{(1)},...,x_{k_{m}}^{(m)}\right) \right\Vert ^{p}\right) ^{\frac{1%
}{p}}\leq C\prod_{i=1}^{m}\left\Vert \left( x_{k_{i}}^{(i)}\right)
_{k_{i}=1}^{n}\right\Vert _{w,q}  \label{t}
\end{equation}%
for all positive integers $n$ and all $x_{k}^{(i)}\in E_{i}$, with $1\leq
k\leq n$ and$\,1\leq i\leq m$. The vector space of all multiple $(p,q)$%
-summing operators is denoted by $\Pi _{(p,q)}^{mult}\left( E_{1},\dots
,E_{m};F\right) $. The infimum, $\pi _{\left( p,q\right) }^{mult}(T)$, taken
over all possible constants $C$ satisfying (\ref{t}) defines a complete norm
in $\Pi _{\left( p,q\right) }^{mult}\left( E_{1},\dots ,E_{m};F\right) $.
When $E_{1}=\cdots =E_{m}=E$, we write $\Pi _{(p,q)}^{mult}\left(
^{m}E;F\right) $. When $k_{1}=\cdots =k_{m}=k$, we recover the definition of
the class of absolutely $\left( p,q\right) $-summing $m$-linear operators
that will be denoted by $\left( \Pi _{(p,q)}\left( E_{1},\dots
,E_{m};F\right) ;\pi _{\left( p,q\right) }(\cdot )\right) $.

For polynomials, let $\mathcal{P}(^{m}E; F)$ denote the Banach space of $m$-homogeneous polynomials from $E$ into $F$, we recall that given
$1\leq p,q<\infty ,$ with $p\geq \frac{q}{m},$ a polynomial $P \in \mathcal{P%
}(^{m}E;F)$ is \textit{absolutely} $\left( p,q\right) $\textit{-summing} if
there is a constant $C\geq 0$ such that%
\begin{equation}
\left( \sum_{k=1}^{n}\left\Vert P(x_{k})\right\Vert ^{p}\right) ^{\frac{1}{p}%
}\leq C\left\Vert \left( x_{k}\right) _{k=1}^{n}\right\Vert _{w,q}^{m}
\label{8866}
\end{equation}%
for all positive integers $n$ and all $x_{k}\in E$, with $1\leq k\leq n$. We
denote by $\mathcal{P}_{(p,q)}\left( ^{m}E;F\right) $ the Banach space of
all absolutely $\left( p,q\right) $-summing polynomials from $E$ to $F$.

\bigskip Of course, when (\ref{t}) or (\ref{8866}) is not valid, this means
that such a constant $C$ does not exist. However it is not obvious at a
first glance that there exists a constant $C_{n}$ depending on $n$
satisfying (\ref{t}) or (\ref{8866}), since at least formally it could
happen that varying the vectors $x_{1},....,x_{n}$ the constant could tend
to infinity. But it is not difficult to prove that this is not the case and
when (\ref{t}) or (\ref{8866}) fails there will exist a constant $C_{n}$
that makes the inequality true. We shall also observe that in all cases a
constant $C_{n}=C_{1}n^{s}$ can be found for a certain $s$ depending on $%
p,q,m$. Note that the number $s$ plays the role of a kind of index of (non)
summability: when $s=0$ the operator is multiple $\left( p,q\right) $%
-summing and when $s$ cannot be chosen to be zero, the map is not multiple $%
\left( p,q\right) $-summing and the \textquotedblleft
optimal\textquotedblright\ value of $s$ can be naturally identified as an
index of (non) summability. In this case, as the \textquotedblleft
optimal\textquotedblright\ value of $s$ grows, we can say that more far from
being multiple $\left( p,q\right) $-summing the map is. We shall adopt a
slightly different approach. Instead of defining the index of summability $s$
as we have just remarked we shall define the index of summability of a pair $%
\left( E_{1}\times \cdots \times E_{m},F\right) $ as follows. The following
definition is inspired by the paper \cite{gusss}, where a kind of index of
summability was investigated for Hardy--Littlewood type inequalities.

\begin{definition}
The \textit{multilinear} $m$-\textit{index of} $\left( p,q\right) $-\textit{%
summability} of a pair $\left( E_{1}\times \cdots \times E_{m},F\right) $ is
defined as
\begin{equation*}
\mathbb{\eta }_{(p,q)}^{m-mult}\left( E_{1},...,E_{m};F\right) =\inf
s_{m,p,q},
\end{equation*}%
where $s_{m,p,q}$ satisfies the following:

There is a constant $C\geq 0$ (depending only on $m\text{ and }T$)
satisfying
\begin{equation*}
\left( \sum_{k_{1},...,k_{m}=1}^{n}\left\Vert T\left(
x_{k_{1}}^{(1)},...,x_{k_{m}}^{(m)}\right) \right\Vert ^{p}\right) ^{\frac{1%
}{p}}\leq Cn^{s_{m,p,q}}\prod_{i=1}^{m}\left\Vert \left(
x_{k_{i}}^{(i)}\right) _{k_{i}=1}^{n}\right\Vert _{w,q}
\end{equation*}%
for every $T\in \mathcal{L}(E_{1},...,E_{m};F)$ and all positive integers $n$
and $x_{k_{i}}^{(i)}\in E_{i}$, with $1\leq k_{i}\leq n\text{ and }1\leq
i\leq m$.

When $E_{1}=\cdots =E_{m}=E$, we write $\mathbb{\eta }_{(p,q)}^{m-mult}%
\left( E;F\right) $ instead $\mathbb{\eta }_{(p,q)}^{m-mult}\left(
E,...,E;F\right) $.
\end{definition}

Similarly the \textit{polynomial} $m$-\textit{index of} $\left( p,q\right) $-%
\textit{summability} of a pair of Banach spaces $\left( E,F\right) $ is
defined as
\begin{equation*}
\mathbb{\eta }_{(p,q)}^{m-pol}\left( E,F\right) =\inf s_{m,p,q},
\end{equation*}%
where $s_{m,p,q}$ satisfies the following:

There is a constant $C>0$ (depending only of $m\text{ and }P$) satisfying
\begin{equation*}
\left( \sum_{j=1}^{n}\left\Vert P(x_{j})\right\Vert ^{p}\right) ^{\frac{1}{p}%
}\leq Cn^{s_{m,p,q}}\left\Vert \left( x_{j}\right) _{j=1}^{n}\right\Vert
_{w,q}^{m}
\end{equation*}%
for every $P\in \mathcal{P}(^{m}E;F)$, all positive integers $n$ and all $%
x_{j}\in E$, with $1\leq j\leq n.$

When $m=1$, we have $\Pi _{(p,q)}^{mult}\left( ^{1}E;F\right) =\mathcal{P}%
_{(p,q)}\left( ^{1}E;F\right) =\Pi _{(p,q)}\left( E;F\right) $ and in this
case we will simply write $\mathbb{\eta }_{(p,q)}\left( E;F\right) $.


\section{Basic results}

One of the cornerstones of the theory of absolutely $p$-summing linear
operators is the Dvoretzky--Rogers Theorem. A weak version of this theorem
asserts that if $p\geq 1$ and $E$ is a Banach space, then the identity
operator on $E$, denoted by $id_{E}$, is absolutely $p$-summing if and only
if $E$ is finite dimensional. The main goal of this section is to certify
that the index of summability is always finite. The next result provides the
$2$-summing norm of the identity operator when $E$ is finite dimensional and
will be very important for us:

\begin{theorem} 
(Pietsch \cite{Piet}) \label{normaidentidade} If $E$ is a Banach space and $\dim E=n$,
then $\pi _{2}(id_{E})=\sqrt{n}.$
\end{theorem}

We highlight the following corollary of the above theorem for future
reference. Note that below we extrapolate the notion of absolutely $p$%
-summing operators to $p>0.$

{}

\begin{corollary}
\label{lemaidentidade}Let $0<p<\infty $. If $E$ is a normed space and $\dim
E=n$, then
\begin{equation}
\pi _{p}(id_{E})\leq n^{\max \left\lbrace \frac{1}{p},\frac{1}{2}\right\rbrace }.  \label{lema}
\end{equation}
\end{corollary}

\begin{proof}
Let $0<p<2$ and $r>0$ such that $\frac{1}{p}=\frac{1}{2}+\frac{1}{r}.$ Thus,
given $x_{1},...,x_{n}\in E$ and using H\"{o}lder's Inequality we obtain

\begin{align*}
\left( \sum_{j=1}^{n}\left\Vert id_{E}(x_{j})\right\Vert ^{p}\right) ^{\frac{%
1}{p}}& \leq \left( \sum_{j=1}^{n}\left\Vert id_{E}(x_{j})\right\Vert
^{2}\right) ^{\frac{1}{2}}\cdot \left( \sum_{j=1}^{n}\left\vert 1\right\vert
^{r}\right) ^{\frac{1}{r}} \\
& \leq \pi _{2}(id_{E})\left\Vert (x_{j})_{j=1}^{n}\right\Vert _{w,2}n^{%
\frac{1}{r}} \\
& \overset{\text{Theorem}\ \ref{normaidentidade}}{\leq }n^{\frac{1}{2}+\frac{%
1}{r}}\left\Vert (x_{j})_{j=1}^{n}\right\Vert _{w,p}.
\end{align*}

Therefore
\begin{equation*}
\pi_{p}(id_{E}) \leq n^{\frac{1}{p}}.
\end{equation*}

For the case $p\geq 2$ we use Inclusion Theorem for absolutely $p$-summing
operators (see \cite[Theorem 2.8]{Diestel}) to obtain
\begin{equation*}
\pi _{p}(id_{E})\leq \pi _{2}(id_{E})=n^{\frac{1}{2}}.
\end{equation*}
\end{proof}

\begin{remark}
\label{remark} Of course that if $X$ is a subspace of an $n$-dimensional
normed space $E$, then
\begin{equation*}
\pi _{p}(id_{X})\leq \left( \dim X\right) ^{\max \left\lbrace \frac{1}{p},\frac{1}{2}%
\right\rbrace }\leq n^{\max \left\lbrace \frac{1}{p},\frac{1}{2}\right\rbrace }.
\end{equation*}
\end{remark}

Although multiple $\left( p,q\right) $-summing operators are defined for $%
p,q\geq 1$, the next result is also valid for $p,q>0.$

\begin{proposition}
\label{uniaomsomantenova} Let $0<p<\infty $ and $E_{1},...,E_{m},F$ be
Banach spaces. Then%
\begin{eqnarray*}
\mathbb{\eta }_{(p,p)}^{m-mult}\left( E_{1},...,E_{m};F\right) &\leq &\frac{m%
}{p}\text{ for }0<p\leq 2; \\
\mathbb{\eta }_{(p,p)}^{m-mult}\left( E_{1},...,E_{m};F\right) &\leq &\frac{m%
}{2}\text{ for }p\geq 2.
\end{eqnarray*}
\end{proposition}

\begin{proof}
Let $T\in \mathcal{L}(E_{1},...,E_{m};F),\;x_{k_{i}}^{(i)}\in E_{i}$ and $%
X_{i}=span\left\{ x_{1_{i}}^{(i)},...,x_{n_{i}}^{(i)}\right\} \subset E_{i}$
with $i=1,...,m$ and $k_{i}=1,...,n.$ Then

\begin{align*}
\left( \sum_{k_{1},...,k_{m}=1}^{n}\left\Vert T\left(
x_{k_{1}}^{(1)},...,x_{k_{m}}^{(m)}\right) \right\Vert ^{p}\right) ^{\frac{1%
}{p}}& \leq \left\Vert T\right\Vert \left( \sum_{k_{1}=1}^{n}\left\Vert
x_{k_{1}}^{(1)}\right\Vert ^{p}\right) ^{\frac{1}{p}}\cdots \left(
\sum_{k_{m}=1}^{n}\left\Vert x_{k_{m}}^{(m)}\right\Vert ^{p}\right) ^{\frac{1%
}{p}} \\
& =\left\Vert T\right\Vert \left( \sum_{k_{1}=1}^{n}\left\Vert
id_{X_{1}}\left( x_{k_{1}}^{(1)}\right) \right\Vert ^{p}\right) ^{\frac{1}{p}%
}\cdots \left( \sum_{k_{m}=1}^{n}\left\Vert id_{X_{m}}\left(
x_{k_{m}}^{(m)}\right) \right\Vert ^{p}\right) ^{\frac{1}{p}}.
\end{align*}

Since $id_{X_{i}}$ is absolutely $p$-summing, for each $i=1,...,m,$ we have

\begin{equation*}
\left( \sum_{k_{i}=1}^{n}\left\Vert
id_{X_{i}}\left(x_{k_{i}}^{(i)}\right)\right\Vert ^{p}\right) ^{\frac{1}{p}}
\leq \pi _{p}(id_{X_{i}})\sup_{\psi \in B_{X_{i}^{^{\ast }}}}\left(
\sum_{k_{i}=1}^{n}\left\vert \psi \left(x_{k_{i}}^{(i)}\right)\right\vert
^{p}\right) ^{\frac{1}{p}}.
\end{equation*}

By the Hahn--Banach Theorem, for each $\psi \in X_{i}^{\ast }$ there is an
extension $\bar{\psi}\in E_{i}^{\ast }$ such that $\left\Vert \psi
\right\Vert =\left\Vert \bar{\psi}\right\Vert .$ Thus

\begin{align*}
\left( \sum_{k_{i}=1}^{n}\left\Vert id_{X_{i}}\left( x_{k_{i}}^{(i)}\right)
\right\Vert ^{p}\right) ^{\frac{1}{p}}& \leq \pi _{p}(id_{X_{i}})\sup_{\bar{%
\psi}\in B_{E_{i}^{\ast }}}\left( \sum_{k_{i}=1}^{n}\left\vert \bar{\psi}%
\left( x_{k_{i}}^{(i)}\right) \right\vert ^{p}\right) ^{\frac{1}{p}} \\
& \leq \pi _{p}(id_{X_{i}})\sup_{\varphi \in B_{E_{i}^{\ast }}}\left(
\sum_{k_{i}=1}^{n}\left\vert \varphi \left( x_{k_{i}}^{(i)}\right)
\right\vert ^{p}\right) ^{\frac{1}{p}} \\
& =\pi _{p}(id_{X_{i}})\left\Vert \left( x_{k_{i}}^{(i)}\right)
_{k_{i}=1}^{n}\right\Vert _{w,p},
\end{align*}%
and hence

\begin{align*}
\left( \sum_{k_{1},...,k_{m}=1}^{n}\left\Vert T\left(
x_{k_{1}}^{(1)},...,x_{k_{m}}^{(m)}\right) \right\Vert ^{p}\right) ^{\frac{1%
}{p}}& \leq \left\Vert T\right\Vert \pi _{p}(id_{X_{1}})\left\Vert \left(
x_{k_{1}}^{(1)}\right) _{k_{1}=1}^{n}\right\Vert _{w,p}\cdots \pi
_{p}(id_{X_{m}})\left\Vert \left( x_{k_{m}}^{(m)}\right)
_{k_{m}=1}^{n}\right\Vert _{w,p} \\
& \leq \left\Vert T\right\Vert \prod_{i=1}^{m}\left( \pi
_{p}(id_{X_{i}})\left\Vert \left( x_{k_{i}}^{(i)}\right)
_{k_{i}=1}^{n}\right\Vert _{w,p}\right) .
\end{align*}

By the previous corollary, we have:

\item[1)] If $0<p\leq 2,$ then

\begin{align*}
\left( \sum_{k_{1},...,k_{m}=1}^{n}\left\Vert T\left(
x_{k_{1}}^{(1)},...,x_{k_{m}}^{(m)}\right) \right\Vert ^{p}\right) ^{\frac{1%
}{p}}& \overset{(\ref{lema})}{\leq }\left\Vert T\right\Vert \left( n^{\frac{1%
}{p}}\right) ^{m}\prod_{i=1}^{m}\left\Vert \left( x_{k_{i}}^{(i)}\right)
_{k_{i}=1}^{n}\right\Vert _{w,p} \\
& =\left\Vert T\right\Vert n^{\frac{m}{p}}\prod_{i=1}^{m}\left\Vert \left(
x_{k_{i}}^{(i)}\right) _{k_{i}=1}^{n}\right\Vert _{w,p},
\end{align*}%
and

\begin{equation*}
\mathbb{\eta }_{(p,p)}^{m-mult}\left( E_{1}, \cdots , E_{m};F\right) \leq
\frac{m}{p}.
\end{equation*}

\item[2)] If $p\geq 2,$ then, analogously,

\begin{equation*}
\mathbb{\eta }_{(p,p)}^{m-mult}\left( E_{1}, \cdots , E_{m};F\right) \leq
\frac{m}{2}.
\end{equation*}
\end{proof}

The next result shows that the above estimates can not be improved, keeping
its universality.

\begin{corollary}
$\mathbb{\eta }_{(2,2)}^{m-mult}\left( \ell _{2};c_{0}\right) = \frac{m}{2}$.
\end{corollary}

\begin{proof}
Let $t$ be a positive real number such that for each $T\in \mathcal{L}%
(^{m}\ell _{2};c_{0})$ there is a constant $C\geq 0$ such that
\begin{equation}
\left( \sum_{k_{1},...,k_{m}=1}^{n}\left\Vert T\left(
x_{k_{1}}^{(1)},...,x_{k_{m}}^{(m)}\right) \right\Vert ^{2}\right) ^{\frac{1%
}{2}}\leq Cn^{t}\prod_{i=1}^{m}\left\Vert \left( x_{k_{i}}^{(i)}\right)
_{k_{i}=1}^{n}\right\Vert _{w,2}  \label{joma}
\end{equation}%
for all positive integers $n$ and all $x_{k_{i}}^{(i)}\in \ell _{2}$, with $%
1\leq k_{i}\leq n$.

Now, let $T\in \mathcal{L}(^{m}\ell _{2};c_{0})$ be defined by
\begin{equation*}
T\left( x^{(1)},...,x^{(m)}\right) =\left( x_{j_{1}}^{(1)}\cdots
x_{j_{m}}^{(m)}\right) _{j_{1},...,j_{m}=1}^{n}.
\end{equation*}%
Of course $\left\Vert T\right\Vert =1$ and
\begin{equation*}
\left( \sum_{j_{1},...,j_{m}=1}^{n}\left\Vert
T(e_{j_{1}},...,e_{j_{m}})\right\Vert ^{2}\right) ^{\frac{1}{2}}=n^{\frac{m}{%
2}}.
\end{equation*}%
Since $\left\Vert (e_{j_{i}})_{j_{i}=1}^{n}\right\Vert _{w,2}=1,$ the latter
condition together with (\ref{joma}) imply
\begin{equation*}
n^{\frac{m}{2}}\leq Cn^{t}
\end{equation*}%
and thus $t\geq \frac{m}{2}$. The converse inequality is given by the
previous proposition and the proof is done.
\end{proof}


\bigskip If $q<p$ it is plain that
\begin{equation*}
\mathbb{\eta }_{(p,q)}^{m-mult}\left( E_{1},...,E_{m};F\right) \leq \mathbb{%
\eta }_{(p,p)}^{m-mult}\left( E_{1},...,E_{m};F\right) .
\end{equation*}%
The next results provide better estimates.\bigskip

\begin{proposition}
\label{p>q} Let $1\leq q\leq p<\infty $ and $E_{1},...,E_{m},F$ be Banach
spaces. Then%
\begin{eqnarray*}
\mathbb{\eta }_{(p,q)}^{m-mult}\left( E_{1},...,E_{m};F\right) &\leq &\frac{m%
}{p}\text{ for }1\leq q\leq 2; \\
\mathbb{\eta }_{(p,q)}^{m-mult}\left( E_{1},...,E_{m};F\right) &\leq &\frac{%
mq}{2p}\text{ for }q\geq 2.
\end{eqnarray*}
\end{proposition}

\begin{proof}
Note that

\begin{equation*}
\left( \sum_{k_{1},...,k_{m}=1}^{n}\left\Vert T\left(
x_{k_{1}}^{(1)},...,x_{k_{m}}^{(m)}\right) \right\Vert ^{p}\right) ^{\frac{1%
}{p}}\leq \left\Vert T\right\Vert \left( \sum_{k_{1}=1}^{n}\left\Vert
x_{k_{1}}^{(1)}\right\Vert ^{p}\right) ^{\frac{1}{p}}\cdots \left(
\sum_{k_{m}=1}^{n}\left\Vert x_{k_{m}}^{(m)}\right\Vert ^{p}\right) ^{\frac{1%
}{p}}.
\end{equation*}

Let $X_{i}:=span\left\{ x_{1_{i}}^{(i)},...,x_{n_{i}}^{(i)}\right\} \subset
E_{i}$ with $i=1,...,m.$ Since $X_{i}$ is a finite dimensional Banach space it follows that $id_{X_{i}}$ is absolutely $q$-summing. So, by \cite[Corollary 16.3.1]{Garling} we have

\begin{equation}  \label{idpq}
\pi_{p,q}(id_{X_{i}}) \leq \pi_{q}(id_{X_{i}})^{\frac{q}{p}}.
\end{equation}

Thus, for each $i=1,...,m,$ we obtain

\begin{equation*} 
\left( \sum_{k_{i}=1}^{n}\left\Vert x_{k_{i}}^{(i)}\right\Vert ^{p}\right) ^{%
\frac{1}{p}} \leq \pi_{p,q}(id_{X_{i}}) \left\Vert (x_{k_{i}})_{k_{i}=1}^{n}
\right\Vert_{w,q} \overset{(\ref{idpq})}{\leq} \pi_{q}(id_{X_{i}})^{\frac{q}{%
p}} \left\Vert (x_{k_{i}})_{k_{i}=1}^{n} \right\Vert_{w,q}
\end{equation*}

and, for $q\geq 2,$ we have

\begin{equation*}
\left( \sum_{k_{i}=1}^{n}\left\Vert x_{k_{i}}^{(i)}\right\Vert ^{p}\right) ^{%
\frac{1}{p}}\leq \left( n^{\frac{1}{2}}\right) ^{\frac{q}{p}}\left\Vert
(x_{k_{i}})_{k_{i}=1}^{n}\right\Vert _{w,q}=n^{\frac{q}{2p}}\left\Vert
(x_{k_{i}})_{k_{i}=1}^{n}\right\Vert _{w,q}.
\end{equation*}

Therefore%
\begin{equation*}
\mathbb{\eta }_{(p,q)}^{m-mult}\left( E_{1},...,E_{m};F\right) \leq \frac{mq%
}{2p}.
\end{equation*}

Analogously, when $1\leq q\leq 2$ we conclude that

\begin{equation*}
\left( \sum_{k_{1},...,k_{m}=1}^{n}\left\Vert T\left(
x_{k_{1}}^{(1)},...,x_{k_{m}}^{(m)}\right) \right\Vert ^{p}\right) ^{\frac{1%
}{p}}=\left\Vert T\right\Vert n^{\frac{m}{p}}\prod_{i=1}^{m}\left\Vert
\left( x_{k_{i}}^{(i)}\right) _{k_{i}=1}^{n}\right\Vert _{w,q}
\end{equation*}%
and%
\begin{equation*}
\mathbb{\eta }_{(p,q)}^{m-mult}\left( E_{1},...,E_{m};F\right) \leq \frac{m}{%
p}.
\end{equation*}
\end{proof}


\bigskip It is well-known that the notion of multiple $\left( p;q\right) $%
-summing operators has no sense when $p<q$, because just the null map would
satisfy the definition. But, curiously, in our context it makes sense to
extrapolate the definition to $0<p<q.$

\begin{proposition}
\label{p<q} Let $0<p<q<\infty $ and $E_{1},...,E_{m},F$ be Banach spaces.
Then%
\begin{eqnarray*}
\mathbb{\eta }_{(p,q)}^{m-mult}\left( E_{1},...,E_{m};F\right) &\leq &\frac{m%
}{p}\text{ for }0<q\leq 2; \\
\mathbb{\eta }_{(p,q)}^{m-mult}\left( E_{1},...,E_{m};F\right) &\leq &\frac{%
(qp-2p+2q)m}{2qp}\text{ for }q\geq 2.
\end{eqnarray*}
\end{proposition}

\begin{proof}
Note that

\begin{equation*}
\left( \sum_{k_{1},...,k_{m}=1}^{n}\left\Vert T\left(
x_{k_{1}}^{(1)},...,x_{k_{m}}^{(m)}\right) \right\Vert ^{p}\right) ^{\frac{1%
}{p}}\leq \left\Vert T\right\Vert \left( \sum_{k_{1}=1}^{n}\left\Vert
x_{k_{1}}^{(1)}\right\Vert ^{p}\right) ^{\frac{1}{p}}...\left(
\sum_{k_{m}=1}^{n}\left\Vert x_{k_{m}}^{(m)}\right\Vert ^{p}\right) ^{\frac{1%
}{p}}.
\end{equation*}%
For all $i=1,...,m,$ the H\"{o}lder inequality tells us that

\begin{align*}
\left( \sum_{k_{i}=1}^{n}\left\Vert x_{k_{i}}^{(i)}\right\Vert ^{p}\right) ^{%
\frac{1}{p}}& \leq \left( \sum_{k_{i}=1}^{n}\left\Vert
x_{k_{i}}^{(i)}\right\Vert ^{q}\right) ^{\frac{1}{q}}\left(
\sum_{k_{m}=1}^{n}\left\vert 1\right\vert ^{\frac{pq}{q-p}}\right) ^{\frac{%
q-p}{pq}} \\
& \leq \left( \sum_{k_{i}=1}^{n}\left\Vert x_{k_{i}}^{(i)}\right\Vert
^{q}\right) ^{\frac{1}{q}}n^{\frac{q-p}{qp}}.
\end{align*}

Hence, for $q\geq 2,$ we have

\begin{align*}
\left( \sum_{k_{1},...,k_{m}=1}^{n}\left\Vert T\left(
x_{k_{1}}^{(1)},...,x_{k_{m}}^{(m)}\right) \right\Vert ^{p}\right) ^{\frac{1%
}{p}}& \overset{(\ref{lema})}{\leq }\left\Vert T\right\Vert \left( n^{\frac{1%
}{2}}\left\Vert \left( x_{k_{1}}^{(1)}\right) _{k_{1}=1}^{n}\right\Vert
_{w,q}n^{\frac{q-p}{qp}}\right) \cdots \left( n^{\frac{1}{2}}\left\Vert
\left( x_{k_{m}}^{(m)}\right) _{k_{m}=1}^{n}\right\Vert _{w,q}n^{\frac{q-p}{%
qp}}\right) \\
& =\left\Vert T\right\Vert n^{\frac{(qp-2p+2q)m}{2qp}}\prod_{i=1}^{m}\left%
\Vert \left( x_{k_{i}}^{(i)}\right) _{k_{i}=1}^{n}\right\Vert _{w,q}
\end{align*}

and, for $0<q\leq 2,$ we get

\begin{align*}
\left( \sum_{k_{1},...,k_{m}=1}^{n}\left\Vert T\left(
x_{k_{1}}^{(1)},...,x_{k_{m}}^{(m)}\right) \right\Vert ^{p}\right) ^{\frac{1%
}{p}}& \overset{(\ref{lema})}{\leq }\left\Vert T\right\Vert \left( n^{\frac{1%
}{q}}\left\Vert \left( x_{k_{1}}^{(1)}\right) _{k_{1}=1}^{n}\right\Vert
_{w,q}n^{\frac{q-p}{qp}}\right) \cdots \left( n^{\frac{1}{q}}\left\Vert
\left( x_{k_{m}}^{(m)}\right) _{k_{m}=1}^{n}\right\Vert _{w,q}n^{\frac{q-p}{%
qp}}\right) \\
& =\left\Vert T\right\Vert n^{\frac{m}{p}}\prod_{i=1}^{m}\left\Vert \left(
x_{k_{i}}^{(i)}\right) _{k_{i}=1}^{n}\right\Vert _{w,q}.
\end{align*}
\end{proof}

\bigskip It is plain that the polynomial $m$-index of $\left( p,q\right) $%
-summability can be estimated using the estimates for the multilinear $m$%
-index of $\left( p,q\right) $-summability. Below we present more accurate
estimates.


\begin{proposition}
\label{pol} Let $E,F$ be Banach spaces, $m$ be a natural number, $q>0$ and $%
p<\frac{q}{m}$. Then
\begin{eqnarray*}
\mathbb{\eta }_{(p,q)}^{m-pol}\left( E;F\right) &\leq &\frac{1}{p}\text{ for
}0<q\leq 2; \\
\mathbb{\eta }_{(p,q)}^{m-pol}\left( E;F\right) &\leq &\frac{1}{p}+\frac{%
m(q-2)}{2q}\text{ for }q\geq 2.
\end{eqnarray*}
\end{proposition}

\begin{proof}
For any $P\in \mathcal{P}\left( ^{m}E;F\right) ,$ by virtue of the H\"{o}%
lder inequality we have
\begin{align*}
\left( \sum\limits_{k=1}^{n}\left\Vert P\left( x_{k}\right) \right\Vert
^{p}\right) ^{\frac{1}{p}}& \leq \left\Vert P\right\Vert \left(
\sum_{k=1}^{n}\left\Vert x_{k}\right\Vert ^{mp}\right) ^{\frac{1}{p}} \\
& \leq \left\Vert P\right\Vert \left[ \left( \sum_{k=1}^{n}\left( \left\Vert
x_{k}\right\Vert ^{mp}\right) ^{\frac{q}{mp}}\right) ^{\frac{mp}{q}}\left(
\sum_{k=1}^{n}\left\vert 1\right\vert ^{\left( \frac{q}{mp}\right) ^{\ast
}}\right) ^{\frac{1}{\left( \frac{q}{mp}\right) ^{\ast }}}\right] ^{\frac{1}{%
p}} \\
& =\left\Vert P\right\Vert \left( \sum_{k=1}^{n}\left\Vert x_{k}\right\Vert
^{q}\right) ^{\frac{m}{q}}n^{\frac{q-mp}{qp}}.
\end{align*}%
Hence, for $0<q\leq 2,$ we have
\begin{align*}
\left( \sum\limits_{k=1}^{n}\left\Vert P\left( x_{k}\right) \right\Vert
^{p}\right) ^{\frac{1}{p}}& \leq \left\Vert P\right\Vert n^{\frac{m}{q}}n^{%
\frac{q-mp}{qp}}\left\Vert (x_{k})_{k=1}^{n}\right\Vert _{w,q}^{m} \\
& =\left\Vert P\right\Vert n^{\frac{1}{p}}\left\Vert
(x_{k})_{k=1}^{n}\right\Vert _{w,q}^{m},
\end{align*}%
and
\begin{equation*}
\mathbb{\eta }_{(p,q)}^{m-pol}\left( E;F\right) \leq \frac{1}{p}.
\end{equation*}%
For $q\geq 2$ we obtain
\begin{align*}
\left( \sum\limits_{k=1}^{n}\left\Vert P\left( x_{k}\right) \right\Vert
^{p}\right) ^{\frac{1}{p}}& \leq \left\Vert P\right\Vert n^{\frac{m}{2}}n^{%
\frac{q-mp}{qp}}\left\Vert (x_{k})_{k=1}^{n}\right\Vert _{w,q}^{m} \\
& =\left\Vert P\right\Vert n^{\frac{mpq+2q-2pm}{2pq}}\left\Vert
(x_{k})_{k=1}^{n}\right\Vert _{w,q}^{m}
\end{align*}%
and thus
\begin{equation*}
\mathbb{\eta }_{(p,q)}^{m-pol}\left( E;F\right) \leq \frac{1}{p}+\frac{m(q-2)%
}{2q}.
\end{equation*}
\end{proof}



\section{Main results: vector-valued maps}

We begin this section with a (simple) technical lemma.

\begin{lemma}
\label{lemar} Let $E$ be an $n$-dimensional Banach space. If $1 \leq d \leq s
\leq 2,$ then there exists a constant $K>0$ such that
\begin{equation*}
Kn^{\frac{2d+s(d-2)}{2sd}} \leq \pi_{s,d}^{(n)}(id_{E}).
\end{equation*}
\end{lemma}

\begin{proof}
Using the Inclusion Theorem \cite[Theorem 10.4]{Diestel} we have

\begin{equation*}
\pi _{\frac{2sd}{2d+s(d-2)},2}^{(n)}(id_{E})\leq \pi _{s,d}^{(n)}(id_{E})
\end{equation*}%
and by invoking Theorem \ref{szarek} we know that there is a constant $C>0$
such that
\begin{equation*}
\frac{1}{C}\pi _{\frac{2sd}{2d+s(d-2)},2}(id_{E})\leq \pi _{\frac{2sd}{%
2d+s(d-2)},2}^{(n)}(id_{E}).
\end{equation*}%
Theorem \ref{konig} assures the existence of a constant $A>0$ such that
\begin{equation*}
An^{\frac{1}{\frac{2sd}{2d+s(d-2)}}}\leq \pi _{\frac{2sd}{2d+s(d-2)}%
,2}(id_{E}).
\end{equation*}%
Therefore
\begin{equation*}
Kn^{\frac{2d+s(d-2)}{2sd}}\leq \pi _{s,d}^{(n)}(id_{E}),
\end{equation*}%
where $K=A/C$.
\end{proof}

We recall that for $2\leq q\leq \infty $, a Banach space $E$ has cotype $q$
if there is a constant $C\geq 0$ such that no matter how we select finitely
many vectors $x_{1},...,x_{n}$ from $E$,
\begin{equation*}
\left( \sum_{k=1}^{n}\left\Vert x_{k}\right\Vert ^{q}\right) ^{\frac{1}{q}%
}\leq C\left( \int_{0}^{1}\left\Vert \sum_{k=1}^{n}r_{k}(t)x_{k}\right\Vert
^{2}dt\right) ^{\frac{1}{2}},
\end{equation*}%
where $r_{k}$ denotes the $k$-th Rademacher function, that is, given $k\in
\mathbb{N}\text{ and }t\in \left[ 0,1\right] ,$ we have $r_{k}(t)=\mathrm{%
sign}\left[ \mathrm{sin}\left( 2^{k}\pi t\right) \right] .$ When $q=\infty ,$
the left hand side will be replaced by the sup norm. It is plain that if $%
q_{1}\leq q_{2},$ then $E$ has cotype $q_{1}$ implies that $E$ has cotype $%
q_{2};$ thus, henceforth, we will denote $\inf \{q:E\mbox{ has cotype }q\}$
by $\cot (E)$.

Now we state and prove the main result of this section. The arguments are
based in ideas taken from \cite{bbr, lin}:

\begin{theorem}
\label{cotipo} Let $E,F$ be infinite dimensional Banach spaces and $r:=\cot
\left( F\right) .$

\begin{enumerate}
\item[(a)] For $1 \leq q \leq 2$ and $0< p \leq \frac{rq}{mr+q}$, we have
\begin{equation*}
\frac{m}{2} \leq \eta _{(p,q)}^{m-pol}\left( E;F\right).
\end{equation*}

\item[(b)] For $1 \leq q \leq 2$ and $\frac{rq}{mr+q} \leq p \leq \frac{2r}{mr+2}$%
, we have
\begin{equation*}
\frac{mp+2}{2p}-\frac{mr+q}{rq} \leq \eta _{(p,q)}^{m-pol}\left( E;F\right).
\end{equation*}

\item[(c)] For $2 \leq q < \infty \text{ and } 0 < p \leq \frac{2r}{mr+2},$ we
have
\begin{equation*}
\frac{m}{2}\leq \eta _{(p,q)}^{m-pol}\left( E;F\right).
\end{equation*}

\item[(d)] For $2 \leq q < \infty \text{ and } \frac{2r}{mr+2} < p < r,$ we
have
\begin{equation*}
\frac{r-p}{pr}\leq \eta _{(p,q)}^{m-pol}\left( E;F\right).
\end{equation*}
\end{enumerate}
\end{theorem}

\begin{proof}
Since $F$ is infinite dimensional, from \cite[Theorem 14.5]{Diestel} we have
\begin{equation*}
\mathrm{cot}(F)=\sup \{2\leq s\leq \infty :F\text{ finitely factors the
formal inclusion }\ell _{s}\hookrightarrow \ell _{\infty }\},
\end{equation*}%
and from \cite[p.304]{Diestel} we know that this supremum is attained. So $F$
finitely factors the formal inclusion $\ell _{r}\hookrightarrow \ell
_{\infty },$ that is, there exist $C_{1},C_{2}>0$ such that for every $n\in
\mathbb{N},$ there are $y_{1},...,y_{n}\in F$ so that
\begin{equation}
C_{1}\left\Vert \left( a_{j}\right) _{j=1}^{n}\right\Vert _{\infty }\leq
\left\Vert \sum_{j=1}^{n}a_{j}y_{j}\right\Vert \leq C_{2}\left(
\sum_{j=1}^{n}\left\vert a_{j}\right\vert ^{r}\right) ^{\frac{1}{r}}
\label{eqcot}
\end{equation}%
for every $a_{1},...,a_{n}\in \mathbb{K}.$

Consider $x_{1}^{\ast },\ldots ,x_{n}^{\ast }\in B_{E^{\ast }}$ such that $%
x_{j}^{\ast }(x_{j})=\left\Vert x_{j}\right\Vert $ for every $j=1,\ldots ,n$%
. Let $a_{1},\ldots ,a_{n}$ be scalars such that $\sum\limits_{j=1}^{n}\left%
\vert a_{j}\right\vert ^{\frac{r}{p}}=1$ and define
\begin{equation*}
P_{n}\colon E\longrightarrow F~,~P_{n}(x)=\sum\limits_{j=1}^{n}\left\vert
a_{j}\right\vert ^{\frac{1}{p}}x_{j}^{\ast }(x)^{m}y_{j}.
\end{equation*}%
Then for every $x\in E$, by (\ref{eqcot})
\begin{equation*}
\left\Vert P_{n}(x)\right\Vert =\left\Vert \sum\limits_{j=1}^{n}\left\vert
a_{j}\right\vert ^{\frac{1}{p}}x_{j}^{\ast }(x)^{m}y_{j}\right\Vert \leq
C_{2}\left( \sum\limits_{j=1}^{n}\left\vert \left\vert a_{j}\right\vert ^{%
\frac{1}{p}}x_{j}^{\ast }(x)^{m}\right\vert ^{r}\right) ^{\frac{1}{r}}\leq
C_{2}\left( \sum\limits_{j=1}^{n}\left\vert a_{j}\right\vert ^{\frac{r}{p}%
}\right) ^{\frac{1}{r}}\left\Vert x\right\Vert ^{m}=C_{2}\left\Vert
x\right\Vert ^{m},
\end{equation*}%
and thus
\begin{equation}
\left\Vert P_{n}\right\Vert \leq C_{2}.  \label{pn}
\end{equation}%
Note that for $k=1,\ldots ,n$, from (\ref{eqcot}), we have
\begin{equation}
\left\Vert P_{n}(x_{k})\right\Vert =\left\Vert
\sum\limits_{j=1}^{n}\left\vert a_{j}\right\vert ^{\frac{1}{p}}x_{j}^{\ast
}(x_{k})^{m}y_{j}\right\Vert \geq C_{1}\left\Vert \left( \left\vert
a_{j}\right\vert ^{\frac{1}{p}}x_{j}^{\ast }(x_{k})^{m}\right)
_{j=1}^{n}\right\Vert _{\infty }\geq C_{1}\left\vert a_{k}\right\vert ^{%
\frac{1}{p}}x_{k}^{\ast }(x_{k})^{m}=C_{1}\left\vert a_{k}\right\vert ^{%
\frac{1}{p}}\left\Vert x_{k}\right\Vert ^{m}.  \label{766}
\end{equation}%
Hence,
\begin{align*}
\left( \sum\limits_{j=1}^{n}\left\Vert x_{j}\right\Vert ^{mp}\left\vert
a_{j}\right\vert \right) ^{\frac{1}{p}}& =\left( \sum\limits_{j=1}^{n}\left(
\left\Vert x_{j}\right\Vert ^{m}\left\vert a_{j}\right\vert ^{\frac{1}{p}%
}\right) ^{p}\right) ^{\frac{1}{p}} \\
& =\frac{1}{C_{1}}\left( \sum\limits_{j=1}^{n}\left( C_{1}\left\Vert
x_{j}\right\Vert ^{m}\left\vert a_{j}\right\vert ^{\frac{1}{p}}\right)
^{p}\right) ^{\frac{1}{p}} \\
& \overset{(\ref{766})}{\leq }\frac{1}{C_{1}}\left(
\sum\limits_{j=1}^{n}\left\Vert P_{n}(x_{j})\right\Vert ^{p}\right) ^{\frac{1%
}{p}}.
\end{align*}%
Suppose that there exists $t\geq 0\text{ and }D>0$ such that%
\begin{equation*}
\left( \sum\limits_{j=1}^{n}\left\Vert P_{n}\left( x_{j}\right) \right\Vert
^{p}\right) ^{\frac{1}{p}}\leq Dn^{t}\left\Vert P_{n}\right\Vert \left\Vert
\left( x_{j}\right) _{j=1}^{n}\right\Vert _{w,q}^{m},
\end{equation*}%
hence
\begin{equation}
\left( \sum\limits_{j=1}^{n}\left\Vert x_{j}\right\Vert ^{mp}\left\vert
a_{j}\right\vert \right) ^{\frac{1}{p}}\leq \frac{D}{C_{1}}\left\Vert
P_{n}\right\Vert n^{t}\left\Vert (x_{j})_{j=1}^{n}\right\Vert _{w,q}^{m}.
\label{123}
\end{equation}%
Since this last inequality holds whenever $\sum\limits_{j=1}^{n}\left\vert
a_{j}\right\vert ^{\frac{r}{p}}=1$ and $p<r,$ we have

\begin{align*}
\left( \sum\limits_{j=1}^{n}\left\Vert x_{j}\right\Vert ^{mp\left( \frac{r}{p%
}\right) ^{\ast }}\right) ^{\frac{1}{\left( \frac{r}{p}\right) ^{\ast }}}& =
\sup \left\{ \left\vert \sum\limits_{j=1}^{n} a_{j}\left\Vert
x_{j}\right\Vert ^{mp} \right\vert :\sum\limits_{j=1}^{n}\vert a_{j}\vert ^{%
\frac{r}{p}}=1\right\} \\
& \leq \sup \left\{ \sum\limits_{j=1}^{n}\left\vert a_{j}\right\vert
\left\Vert x_{j}\right\Vert ^{mp}:\sum\limits_{j=1}^{n}\left\vert
a_{j}\right\vert ^{\frac{r}{p}}=1\right\} \\
& \overset{(\ref{123})}{\leq} \left( \frac{D}{C_{1}} \left\Vert P_{n}\right
\Vert n^{t}\left\Vert (x_{j})_{j=1}^{n}\right\Vert _{w,q}^{m}\right) ^{p} \\
& \overset{(\ref{pn})}{\leq} \left( \frac{DC_{2}}{C_{1}} n^{t}\left\Vert
(x_{j})_{j=1}^{n}\right\Vert _{w,q}^{m}\right) ^{p}
\end{align*}
and thus, denote $\frac{DC_{2}}{C_{1}} := Q$
\begin{equation*}
\frac{\left( \sum\limits_{j=1}^{n}\left\Vert x_{j}\right\Vert ^{mp\left(
\frac{r}{p}\right) ^{\ast }}\right) ^{\frac{1}{\left( \frac{r}{p}\right)
^{\ast }}}}{\left\Vert (x_{j})_{j=1}^{n}\right\Vert _{w,q}^{mp}}\leq
n^{tp}Q^{p}.
\end{equation*}

Therefore%
\begin{equation}  \label{ge}
\frac{\left( \sum\limits_{j=1}^{n}\left\Vert x_{j}\right\Vert ^{mp\left(
\frac{r}{p}\right) ^{\ast }}\right) ^{\frac{1}{mp\left( \frac{r}{p}\right)
^{\ast }}}}{\left\Vert (x_{j})_{j=1}^{n}\right\Vert _{w,q}}\leq n^{\frac{t}{m%
}}Q^{\frac{1}{m}}.
\end{equation}

Note that (\ref{ge}) is valid for any $x_{1},...,x_{n}$. So, for any $n$%
-dimensional subspace $X$ of $E$ we have
\begin{equation}  \label{yyy}
\frac{\left( \sum\limits_{j=1}^{n}\left\Vert id_{X}(x_{j})\right\Vert
^{mp\left( \frac{r}{p}\right) ^{\ast }}\right) ^{\frac{1}{mp\left( \frac{r}{p%
}\right) ^{\ast }}}}{\left\Vert (x_{j})_{j=1}^{n}\right\Vert _{w,q}}\leq n^{%
\frac{t}{m}}Q^{\frac{1}{m}},
\end{equation}
for all $x_{1},...,x_{n}\in X$.

(a) Since
\begin{equation*}
0<p\leq \frac{rq}{mr+q},
\end{equation*}%
we have
\begin{equation*}
mp\left( \frac{r}{p}\right) ^{\ast }\leq q,
\end{equation*}%
and
\begin{equation*}
\frac{\left( \sum\limits_{j=1}^{n}\left\Vert id_{X}(x_{j})\right\Vert
^{q}\right) ^{\frac{1}{q}}}{\left\Vert (x_{j})_{j=1}^{n}\right\Vert _{w,q}}%
\leq n^{\frac{t}{m}}Q^{\frac{1}{m}}. 
\end{equation*}

So
\begin{equation*}
\pi_{q}^{(n)}(id_{X})\leq n^{\frac{t}{m}}Q^{\frac{1}{m}}.
\end{equation*}

Since $q\leq 2$, by \cite[Theorem 2.8]{Diestel} we get

\begin{equation}  \label{e}
\pi_{2}^{(n)}(id_{X}) \leq n^{\frac{t}{m}}Q^{\frac{1}{m}}.
\end{equation}

Now Theorem \ref{szarek} assures us that there is a constant $C>0$ such that
\begin{equation}
\pi _{2}(id_{X})\leq C\pi _{2}^{(n)}(id_{X}).  \label{e20}
\end{equation}%
Using (\ref{e}), (\ref{e20}) and Theorem \ref{normaidentidade} we obtain
\begin{equation*}
\frac{1}{C}n^{1/2}\leq n^{t/m}Q^{\frac{1}{m}}.
\end{equation*}%
Thus%
\begin{equation*}
t\geq \frac{m}{2}.
\end{equation*}%
Therefore
\begin{equation*}
\eta _{(p,q)}^{m-pol}\left( E;F\right) \geq \frac{m}{2}.
\end{equation*}%
\bigskip

(b) By (\ref{yyy}), we have
\begin{equation}
\pi _{mp\left( \frac{r}{p}\right) ^{\ast },q}^{(n)}(id_{X})\leq n^{\frac{t}{m%
}}Q^{\frac{1}{m}}.  \label{11a}
\end{equation}

Since $\frac{rq}{mr+q} \leq p \leq \frac{2r}{mr+2}$ and $mp\left( \frac{r}{p}%
\right) ^{\ast }=\frac{mpr}{r-p},$ we have $q \leq mp\left( \frac{r}{p}\right)
^{\ast }\leq 2$. From Lemma \ref{lemar}, there is a constant $K>0$ such that
\begin{equation}
Kn^{\frac{2q+mp\left( \frac{r}{p}\right) ^{\ast }(q-2)}{2mp\left( \frac{r}{p}%
\right) ^{\ast }q}} \leq \pi _{mp\left( \frac{r}{p}\right) ^{\ast
},q}^{(n)}(id_{X}).  \label{11}
\end{equation}
From (\ref{11a}) and (\ref{11}) it follows that

\begin{equation*}
Kn^{\frac{2q+mp\left( \frac{2}{p}\right) ^{\ast }(q-2)}{2mp\left( \frac{2}{p}%
\right) ^{\ast }q}}\leq n^{t/m}Q^{\frac{1}{m}}.
\end{equation*}%
Thus%
\begin{equation*}
\frac{t}{m}\geq \frac{mp+2}{2mp}-\frac{mr+q}{mrq}
\end{equation*}%
and we conclude that

\begin{equation*}
t\geq \frac{mp+2}{2p}-\frac{mr+q}{rq}.
\end{equation*}%
Therefore
\begin{equation*}
\eta _{(p,q)}^{m-pol}\left( E;F\right) \geq \frac{mp+2}{2p}-\frac{mr+q}{rq}.
\end{equation*}%
\bigskip

(c) Since $q\geq 2$, we obtain
\begin{equation*}
\frac{\left( \sum\limits_{j=1}^{n}\left\Vert id_{X}(x_{j})\right\Vert
^{mp\left( \frac{r}{p}\right) ^{\ast }}\right) ^{\frac{1}{mp\left( \frac{r}{p%
}\right) ^{\ast }}}}{\left\Vert (x_{j})_{j=1}^{n}\right\Vert _{w,2}}\leq n^{%
\frac{t}{m}}Q^{\frac{1}{m}},
\end{equation*}%
for all $x_{1},...,x_{n}\in X$. But $\frac{2r}{mr+2}\geq p$ implies that $%
mp\left( \frac{r}{p}\right) ^{\ast }\leq 2$, and thus

\begin{equation*}
\frac{\left( \sum\limits_{j=1}^{n}\left\Vert id_{X}(x_{j})\right\Vert
^{2}\right) ^{\frac{1}{2}}}{\left\Vert (x_{j})_{j=1}^{n}\right\Vert _{w,2}}%
\leq n^{\frac{t}{m}}Q^{\frac{1}{m}}.
\end{equation*}%
Therefore
\begin{equation*}
\pi _{2}^{(n)}(id_{X})\leq n^{\frac{t}{m}}Q^{\frac{1}{m}}.
\end{equation*}

From Theorem \ref{szarek} it follows that
\begin{equation*}
\pi _{2}(id_{X})\leq C\pi _{2}^{(n)}(id_{X}).
\end{equation*}%
By Theorem \ref{normaidentidade}, we have

\begin{equation*}
\frac{1}{C}n^{\frac{1}{2}}=\frac{1}{C}\pi _{2}(id_{X})\leq n^{t/m}Q^{\frac{1%
}{m}},
\end{equation*}%
and we conclude that%
\begin{equation*}
t\geq \frac{m}{2},
\end{equation*}%
that is
\begin{equation*}
\eta _{(p,q)}^{m-pol}\left( E;F\right) \geq \frac{m}{2}.
\end{equation*}

\bigskip

(d) Since $q \geq 2$, we have
\begin{equation*}
\frac{\left( \sum\limits_{j=1}^{n}\left\Vert id_{X}(x_{j})\right\Vert
^{mp\left( \frac{r}{p}\right) ^{\ast }}\right) ^{\frac{1}{mp\left( \frac{r}{p%
}\right) ^{\ast }}}}{\left\Vert (x_{j})_{j=1}^{n}\right\Vert _{w,2}}\leq n^{%
\frac{t}{m}}Q ^{\frac{1}{m}},
\end{equation*}
for all $x_{1},...,x_{n}\in X$. Then

\begin{equation*}
\pi_{mp\left( \frac{r}{p}\right) ^{\ast },2}^{(n)}(id_{X}) \leq n^{\frac{t}{m%
}}Q ^{\frac{1}{m}}.  
\end{equation*}

But $\frac{2r}{mr+2}<p$ implies that $mp\left( \frac{r}{p}\right) ^{\ast }>2$%
, and from Theorem \ref{szarek} it follows that
\begin{equation*}
\pi _{mp\left( \frac{r}{p}\right) ^{\ast },2}(id_{X})\leq C\pi _{mp\left(
\frac{r}{p}\right) ^{\ast },2}^{(n)}(id_{X}).
\end{equation*}%
By Theorem \ref{konig}, there is a constant $A>0$ such that

\begin{equation*}
A\cdot n^{\frac{1}{mp\left( \frac{r}{p}\right) ^{\ast }}}\leq \pi _{mp\left(
\frac{r}{p}\right) ^{\ast },2}(id_{X}),
\end{equation*}%
and thus
\begin{equation*}
\frac{A}{C}n^{\frac{r-p}{mpr}}\leq n^{t/m}Q^{\frac{1}{m}}.
\end{equation*}%
Finally, we obtain%
\begin{equation*}
t\geq \frac{r-p}{pr},
\end{equation*}%
and
\begin{equation*}
\eta _{(p,q)}^{m-pol}\left( E;F\right) \geq \frac{r-p}{pr}.
\end{equation*}
\end{proof}

\begin{remark}
In the above result, there is a kind of continuity. In fact, when $p=\frac{rq%
}{mr+q},$ from (a) we have
\begin{equation*}
\frac{m}{2}\leq \eta _{(p,q)}^{m-pol}(E;F).  
\end{equation*}%
On the other hand, considering $p=\frac{%
rq}{mr+q} $ it follows from (b) that
\begin{equation*}
\frac{mp+2}{2p}-\frac{mr+q}{rq} = \frac{m}{2}\leq \eta
_{(p,q)}^{m-pol}(E;F). 
\end{equation*}%

Now, when $p=\frac{2r}{mr+2},$ by (c) we have
\begin{equation}
\frac{m}{2}\leq \eta _{(p,q)}^{m-pol}(E;F).  \label{cc}
\end{equation}%
Given $\epsilon >0$ and taking $p_{\epsilon }=\frac{2r}{mr+2}+\epsilon
$ it follows by (d) that
\begin{equation}
\frac{r-p_{\epsilon }}{p_{\epsilon }r}\leq \eta _{(p_{\epsilon
},q)}^{m-pol}(E;F).  \label{dd}
\end{equation}%
Again, there is a continuity between the lower estimates (\ref{cc}) and (\ref%
{dd}), because letting $\epsilon $ tend to zero, we have
\begin{equation*}
p_{\epsilon }\rightarrow \frac{2r}{mr+2}\ \ \text{ and }\ \ \frac{%
r-p_{\epsilon }}{p_{\epsilon }r}\rightarrow \frac{m}{2}.
\end{equation*}%
The same behavior happens when $q=2$.
\end{remark}

The next two results provide optimality of $\eta _{(p,q)}^{m-pol}(E;F)$ in
some cases:

\begin{corollary}
If $\frac{2}{2m+1} \leq p < \frac{2}{m+1}$, then $\eta_{(p,1)}^{m-pol}(%
\ell_{1};\ell_{2}) = \frac{1}{p}-\frac{m+1}{2}$.
\end{corollary}

\begin{proof}
Considering $q=1 \text{ and }r=2$ in the previous theorem item (b) we have
\begin{equation}  \label{ot1}
\eta_{(p,1)}^{m-pol}(\ell_{1};\ell_{2}) \geq \frac{1}{p}-\frac{m+1}{2}.
\end{equation}

Let us show that (\ref{ot1}) is sharp. From \cite{Bernardino} we know that
every continuous $m$-homogeneous polynomial from $\ell _{1}$ to $\ell _{2}$
is absolutely $(\frac{2}{m+1},1)$-summing. Since $\frac{2}{2m+1} \leq p<\frac{2}{%
m+1}$, let $w>0$ be such that
\begin{equation*}
\frac{1}{p}=\frac{1}{\frac{2}{m+1}}+\frac{1}{w}. 
\end{equation*}%
Given $P\in \mathcal{P}\left( ^{m}\ell _{1};\ell _{2}\right) $, from the H%
\"{o}lder's inequality we have
\begin{equation*}
\left( \sum_{k=1}^{n}\left\Vert P(x_{k})\right\Vert ^{p}\right) ^{\frac{1}{p}%
}\leq n^{\frac{1}{w}}\left( \sum_{k=1}^{n}\left\Vert P(x_{k})\right\Vert ^{%
\frac{2}{m+1}}\right) ^{\frac{m+1}{2}}\leq Dn^{\frac{1}{p}-\frac{m+1}{2}%
}\left\Vert (x_{k})_{k=1}^{n}\right\Vert _{w,1}^{m}.
\end{equation*}
\end{proof}

\begin{corollary}
Let $K$ be a compact Hausdorff space and $F$ be an infinite dimensional
Banach space, with $\mathrm{cot}(F)=r.$ If $\frac{2r}{r+2}<p<r,$ then
\begin{equation*}
\eta _{(p,2)}(C(K);F)=\frac{1}{p}-\frac{1}{r}.
\end{equation*}
\end{corollary}

\begin{proof}
By the previous theorem item (d), if $q=2\text{ and }\mathrm{cot}(F)=r$ we
have
\begin{equation}
\eta _{(p,2)}(C(K);F)\geq \frac{1}{p}-\frac{1}{r}.  \label{ot2}
\end{equation}

Let us show that (\ref{ot2}) is sharp. From \cite[Theorem 11.14]{Diestel} we
know that every continuous linear operator from $C(K)$ to $F$, with $\mathrm{%
cot}(F) = r,$ is absolutely $(r,2)$-summing.

Let $\frac{2r}{r+2}<p<r$ and let $w>0$ be such that
\begin{equation*}
\frac{1}{p}=\frac{1}{r}+\frac{1}{w}.
\end{equation*}%
Given $T\in \mathcal{L}\left( C(K);F\right) $, from the H\"{o}lder's
inequality we have
\begin{equation*}
\left( \sum_{k=1}^{n}\left\Vert T(x_{k})\right\Vert ^{p}\right) ^{\frac{1}{p}%
}\leq n^{\frac{1}{w}}\left( \sum_{k=1}^{n}\left\Vert T(x_{k})\right\Vert
^{r}\right) ^{\frac{1}{r}}\leq Dn^{\frac{1}{p}-\frac{1}{r}}\left\Vert
(x_{k})_{k=1}^{n}\right\Vert _{w,2},
\end{equation*}%
i.e.,
\begin{equation*}
\eta _{(p,2)}(C(K);F)\leq \frac{1}{p}-\frac{1}{r}.
\end{equation*}
\end{proof}


\bigskip

\section{Main results: real-valued maps}

The following result complements the results of the previous section (its
proof is inspired in techniques found in \cite{bpp}); now we consider the
case in which $m$ is even and $F=\mathbb{R}$.

\begin{theorem}
\label{realcommpar} Let $\ m\ $be an even positive integer and $E$ be an
infinite dimensional real Banach space.

\begin{enumerate}
\item[(a)] If $1\leq q\leq 2$ and $0<p\leq \frac{q}{m+q}$, then
\begin{equation*}
\frac{m}{2}\leq \eta _{(p,q)}^{m\text{-}pol}\left( E;\mathbb{R}\right) .
\end{equation*}

\item[(b)] If $1\leq q \leq 2$ and $\frac{q}{m+q}\leq p\leq \frac{2}{m+2}$, then
\begin{equation*}
\frac{mp+2}{2p}-\frac{m+q}{q}\leq \eta _{(p,q)}^{m\text{-}pol}\left( E;%
\mathbb{R}\right) .
\end{equation*}

\item[(c)] If $2\leq q<\infty \text{ and }0<p\leq \frac{2}{m+2},$ then
\begin{equation*}
\frac{m}{2}\leq \eta _{(p,q)}^{m\text{-}pol}\left( E;\mathbb{R}\right) .
\end{equation*}

\item[(d)] If $2\leq q<\infty \text{ and }\frac{2}{m+2}<p<1,$ then
\begin{equation*}
\frac{1-p}{p}\leq \eta _{(p,q)}^{m\text{-}pol}\left( E;\mathbb{R}\right) .
\end{equation*}
\end{enumerate}
\end{theorem}

\begin{proof}
Let $n\in \mathbb{N}\text{ and }x_{1},...,x_{n}\in E$. Consider $x_{1}^{\ast
},\ldots ,x_{n}^{\ast }\in B_{E^{\ast }}$ such that $x_{j}^{\ast
}(x_{j})=\left\Vert x_{j}\right\Vert $ for every $j=1,\ldots ,n$ . Let $%
a_{1},\ldots ,a_{n}$ be real numbers such that $\sum\limits_{j=1}^{n}\left%
\vert a_{j}\right\vert ^{\frac{1}{p}}=1$ and define
\begin{equation*}
P_{n}\colon E\longrightarrow \mathbb{R}~,~P_{n}(x)=\sum\limits_{j=1}^{n}%
\left\vert a_{j}\right\vert ^{\frac{1}{p}}x_{j}^{\ast }(x)^{m},\text{ for
every }x\in E.
\end{equation*}%
Since $m$ is even, it follows that $P(x)\geq 0,\text{ for every }x\in E.$
Hence

\begin{equation*}
\left\vert P_{n}(x)\right\vert =P_{n}(x)=\sum\limits_{j=1}^{n}\left\vert
a_{j}\right\vert ^{\frac{1}{p}}x_{j}^{\ast }(x)^{m}\geq \left\vert
a_{k}\right\vert ^{\frac{1}{p}}x_{k}^{\ast }(x)^{m},\text{ for every }x\in E%
\text{ and }k=1,...,n,
\end{equation*}%
and
\begin{equation}
\left\vert P_{n}(x_{k})\right\vert
=P_{n}(x_{k})=\sum\limits_{j=1}^{n}\left\vert a_{j}\right\vert ^{\frac{1}{p}%
}x_{j}^{\ast }(x_{k})^{m}\geq \left\vert a_{k}\right\vert ^{\frac{1}{p}%
}x_{k}^{\ast }(x_{k})^{m}=\left\vert a_{k}\right\vert ^{\frac{1}{p}%
}\left\Vert x_{k}\right\Vert ^{m},\text{ for }k=1,...,n.  \label{7661}
\end{equation}

Furthermore, for every $x\in E,$ we have

\begin{equation*}
\left\vert P_{n}(x)\right\vert =\left\vert \sum\limits_{j=1}^{n}\left\vert
a_{j}\right\vert ^{\frac{1}{p}}x_{j}^{\ast }(x)^{m}\right\vert \leq \left(
\sum\limits_{j=1}^{n}\left\vert a_{j}\right\vert ^{\frac{1}{p}}\right)
\left\Vert x\right\Vert ^{m}=\left\Vert x\right\Vert ^{m},
\end{equation*}%
and thus
\begin{equation}
\left\Vert P_{n}\right\Vert \leq 1.  \label{pn1}
\end{equation}

Therefore,
\begin{align*}
\left( \sum\limits_{j=1}^{n}\left\Vert x_{j}\right\Vert ^{mp}\left\vert
a_{j}\right\vert \right) ^{\frac{1}{p}}& =\left( \sum\limits_{j=1}^{n}\left(
\left\Vert x_{j}\right\Vert ^{m}\left\vert a_{j}\right\vert ^{\frac{1}{p}%
}\right) ^{p}\right) ^{\frac{1}{p}} \\
& \overset{(\ref{7661})}{\leq }\left( \sum\limits_{j=1}^{n}\left\vert
P_{n}(x_{j})\right\vert ^{p}\right) ^{\frac{1}{p}}.
\end{align*}%
Suppose that there exists $t\geq 0\text{ and }D>0$ such that%
\begin{align*}
\left( \sum\limits_{j=1}^{n}\left\Vert P_{n}\left( x_{j}\right) \right\Vert
^{p}\right) ^{\frac{1}{p}}& \leq D\left\Vert P_{n}\right\Vert
n^{t}\left\Vert (x_{j})_{j=1}^{n}\right\Vert _{w,q}^{m} \\
& \overset{(\ref{pn1})}{\leq }Dn^{t}\left\Vert (x_{j})_{j=1}^{n}\right\Vert
_{w,q}^{m}.
\end{align*}%
Hence
\begin{equation}
\left( \sum\limits_{j=1}^{n}\left\Vert x_{j}\right\Vert ^{mp}\left\vert
a_{j}\right\vert \right) ^{\frac{1}{p}}\leq Dn^{t}\left\Vert
(x_{j})_{j=1}^{n}\right\Vert _{w,q}^{m}.  \label{1231}
\end{equation}%
and since this last inequality holds whenever $\sum\limits_{j=1}^{n}\left%
\vert a_{j}\right\vert ^{\frac{1}{p}}=1$ and $p<1,$ we have

\begin{align*}
\left( \sum\limits_{j=1}^{n}\left\Vert x_{j}\right\Vert ^{\frac{mp}{1-p}%
}\right) ^{1-p}& =\sup \left\{ \left\vert
\sum\limits_{j=1}^{n}a_{j}\left\Vert x_{j}\right\Vert ^{mp}\right\vert
:\sum\limits_{j=1}^{n}\left\vert a_{j}\right\vert ^{\frac{1}{p}}=1\right\} \\
& \leq \sup \left\{ \sum\limits_{j=1}^{n}\left\vert a_{j}\right\vert
\left\Vert x_{j}\right\Vert ^{mp}:\sum\limits_{j=1}^{n}|a_{j}|^{\frac{1}{p}%
}=1\right\} \\
& \overset{(\ref{1231})}{\leq }\left( Dn^{t}\left\Vert
(x_{j})_{j=1}^{n}\right\Vert _{w,q}^{m}\right) ^{p}.
\end{align*}

Therefore
\begin{equation}  \label{gg1}
\frac{\left( \sum\limits_{j=1}^{n}\left\Vert x_{j}\right\Vert ^{\frac{mp}{1-p%
}}\right) ^{\frac{1-p}{mp}}}{\left\Vert (x_{j})_{j=1}^{n}\right\Vert _{w,q}}%
\leq D^{\frac{1}{m}}n^{\frac{t}{m}}.
\end{equation}

See that (\ref{gg1}) is valid for any $x_{1},...,x_{n}.$ So, for any $n$%
-dimensional subspace $X$ of $E$ we have

\begin{equation}  \label{yyy1}
\frac{\left( \sum\limits_{j=1}^{n}\left\Vert id_{X}(x_{j})\right\Vert ^{%
\frac{mp}{1-p}}\right) ^{\frac{1-p}{mp}}}{\left\Vert
(x_{j})_{j=1}^{n}\right\Vert _{w,q}}\leq D^{\frac{1}{m}}n^{\frac{t}{m}}.
\end{equation}
for all $x_{1},...,x_{n} \in X.$

Now we prove each item separately.

(a) Since
\begin{equation*}
0<p\leq \frac{q}{m+q},
\end{equation*}%
we have
\begin{equation*}
\frac{mp}{1-p}\leq q
\end{equation*}

and thus
\begin{equation*}
\frac{\left( \sum\limits_{j=1}^{n}\left\Vert id_{X}(x_{j})\right\Vert
^{q}\right) ^{\frac{1}{q}}}{\left\Vert (x_{j})_{j=1}^{n}\right\Vert _{w,q}}%
\leq D^{\frac{1}{m}}n^{\frac{t}{m}}. 
\end{equation*}
So
\begin{equation*}
\pi _{q}^{(n)}(id_{X})\leq D^{\frac{1}{m}}n^{\frac{t}{m}}.
\end{equation*}%
%
%
%
%
%
%
%
%
%

Since $1\leq q\leq 2$ by \cite[Theorem 2.8]{Diestel} we have

\begin{equation}
\pi _{2}^{(n)}(id_{X})<D^{\frac{1}{m}}n^{\frac{t}{m}},  \label{e2}
\end{equation}%
and from Theorem \ref{szarek} we conclude that
\begin{equation}
\pi _{2}(id_{X})\leq C\pi _{2}^{(n)}(id_{X}).  \label{jo}
\end{equation}%
By Theorem \ref{normaidentidade} we know that
\begin{equation*}
\pi _{2}(id_{X})=n^{1/2}
\end{equation*}%
and thus, from (\ref{e2}) and (\ref{jo}), it follows that
\begin{equation*}
\frac{1}{C}n^{1/2}\leq D^{\frac{1}{m}}n^{t/m}.
\end{equation*}%
Hence
\begin{equation*}
t\geq \frac{m}{2},
\end{equation*}%
i.e.,
\begin{equation*}
\eta _{(p,q)}^{m\text{-}pol}\left( E;\mathbb{R}\right) \geq \frac{m}{2}.
\end{equation*}

\bigskip

(b) By (\ref{yyy1}), we have
\begin{equation}
\pi _{\frac{mp}{1-p},q}^{(n)}(id_{X})\leq n^{\frac{t}{m}}D^{\frac{1}{m}}.
\label{jo1}
\end{equation}

Since $\frac{q}{m+q} \leq p\leq \frac{2}{m+2}$ we have $q \leq \frac{mp}{1-p}\leq 2.$
From (\ref{jo1}) and Lemma \ref{lemar}, there is a constant $K>0$ such that

\begin{equation*}
Kn^{\frac{2q+\frac{mp}{1-p}(q-2)}{2\frac{mp}{1-p}q}}\leq n^{t/m}D^{\frac{1}{m%
}}.
\end{equation*}%
Thus%
\begin{equation*}
\frac{t}{m}\geq \frac{mp+2}{2mp}-\frac{m+q}{mq}
\end{equation*}%
and we conclude that\bigskip

\begin{equation*}
t\geq \frac{mp+2}{2p}-\frac{m+q}{q},
\end{equation*}%
that is,
\begin{equation*}
\eta _{(p,q)}^{m\text{-}pol}\left( E;\mathbb{R}\right) \geq \frac{mp+2}{2p}-%
\frac{m+q}{q}.
\end{equation*}

\bigskip

\item[(c)] Since $q\geq 2$, we have
\begin{equation*}
\frac{\left( \sum\limits_{j=1}^{n}\left\Vert id_{X}(x_{j})\right\Vert ^{%
\frac{mp}{1-p}}\right) ^{\frac{1-p}{mp}}}{\left\Vert
(x_{j})_{j=1}^{n}\right\Vert _{w,2}}\leq D^{\frac{1}{m}}n^{\frac{t}{m}},
\end{equation*}%
for all $x_{1},...,x_{n}\in X$. But $\frac{2}{m+2}\geq p$ implies that $%
\frac{mp}{1-p}\leq 2;$ hence
\begin{equation*}
\frac{\left( \sum\limits_{j=1}^{n}\left\Vert id_{X}(x_{j})\right\Vert
^{2}\right) ^{\frac{1}{2}}}{\left\Vert (x_{j})_{j=1}^{n}\right\Vert _{w,2}}%
\leq D^{\frac{1}{m}}n^{\frac{t}{m}},
\end{equation*}%
and thus
\begin{equation*}
\pi _{2}^{(n)}(id_{X})\leq D^{\frac{1}{m}}n^{\frac{t}{m}}.
\end{equation*}

From Theorem \ref{szarek} we have
\begin{equation*}
\pi _{2}(id_{X})\leq C\pi _{2}^{(n)}(id_{X})
\end{equation*}

and from Theorem \ref{normaidentidade}, we have
\begin{equation*}
\frac{1}{C}n^{\frac{1}{2}}=\frac{1}{C}\pi _{2}(id_{X})\leq n^{t/m}D^{\frac{1%
}{m}}.
\end{equation*}%
So, we conclude that%
\begin{equation*}
t\geq \frac{m}{2}.
\end{equation*}%
and
\begin{equation*}
\eta _{(p,q)}^{m\text{-}pol}\left( E;\mathbb{R}\right) \geq \frac{m}{2}.
\end{equation*}

\bigskip

\item[(d)] Since $q\geq 2,$ we obtain
\begin{equation*}
\frac{\left( \sum\limits_{j=1}^{n}\left\Vert id_{X}(x_{j})\right\Vert ^{%
\frac{mp}{1-p}}\right) ^{\frac{1-p}{mp}}}{\left\Vert
(x_{j})_{j=1}^{n}\right\Vert _{w,2}}\leq D^{\frac{1}{m}}n^{\frac{t}{m}},
\end{equation*}%
for all $x_{1},...,x_{n}\in X$.

So
\begin{equation*}
\pi _{\frac{mp}{1-p},2}^{(n)}(id_{X})\leq n^{\frac{t}{m}}D^{\frac{1}{m}}.
\end{equation*}

But $\frac{2}{m+2}<p$ implies that $\frac{mp}{1-p}>2$, and from Theorem \ref%
{szarek}
\begin{equation*}
\pi _{\frac{mp}{1-p},2}(id_{X})\leq C\pi _{\frac{mp}{1-p},2}^{(n)}(id_{X})
\end{equation*}

By Theorem \ref{konig}, there is a constant $A>0$ such that

\begin{equation*}
An^{\frac{1-p}{mp}}\leq \pi _{\frac{mp}{1-p},2}(id_{X}),
\end{equation*}%
thus
\begin{equation*}
\frac{A}{C}n^{\frac{1-p}{mp}}\leq n^{t/m}D^{\frac{1}{m}},
\end{equation*}%
we conclude that%
\begin{equation*}
t\geq \frac{1-p}{p}.
\end{equation*}%
so
\begin{equation*}
\eta _{(p,q)}^{m\text{-}pol}\left( E;\mathbb{R}\right) \geq \frac{1-p}{p}.
\end{equation*}

\bigskip
\end{proof}

\bigskip

\begin{remark}
As in previous theorem, in this result we have a clear \textquotedblleft
continuity\textquotedblright\ in our estimates.
\end{remark}

\end{document}